\documentclass[11pt]{article}
\usepackage{amsmath,amsthm}
\usepackage{amssymb,enumerate}
\usepackage{graphicx}
\usepackage{array}
\usepackage[svgnames]{xcolor}
\usepackage{url}
\usepackage{fullpage, longtable, nicefrac}

\newtheorem{theorem}{Theorem}[section]
\newtheorem{lemma}[theorem]{Lemma}
\newtheorem{corollary}[theorem]{Corollary}
\newtheorem{proposition}[theorem]{Proposition}

\newcommand\C{{\mathcal C}}
\newcommand\Zw{{\mathbf w}}
\newcommand\Zu{{\mathbf u}}
\newcommand\Zv{{\mathbf v}}

\newcommand\ints{{\mathbb Z}}
\newcommand\re{{\mathbb R}}
\newcommand\rats{{\mathbb Q}}

\newcommand\pmat[1]{\begin{pmatrix} #1 \end{pmatrix}}

\DeclareMathOperator{\Aut}{Aut}

\newcommand\pls[2]{v_{#1, #2}}
\newcommand\mn[2]{w_{#1, #2}}

\title{Simple eigenvalues of cubic vertex-transitive graphs}

\author{Krystal Guo\thanks{Part of this work was done when K. Guo was a  PhD student at Simon Fraser University, see \cite{G2015}. The author was supported in part by the Natural Sciences and Engineering Research Council of Canada (NSERC), Postgraduate Scholarships - Doctoral (PGSD).  } \\[1mm]
   {Korteweg-De Vries Institute}\\{University of Amsterdam} \\
   {Amsterdam, The Netherlands} \\ \texttt{k.guo@uva.nl}
    \and
   Bojan Mohar\thanks{B.M.~was supported in part by the NSERC Discovery Grant R611450 (Canada),
   by the Canada Research Chairs program, and by the Research Project J1-8130 of ARRS (Slovenia).}
   \thanks{On leave from IMFM, Department of Mathematics, University of Ljubljana.}%
 \\[1mm]
 {Department of Mathematics}\\{Simon Fraser University}\\{Burnaby, B.C. V5A 1S6}\\ \texttt{mohar@sfu.ca}
 }
 \date{December 16, 2022}

\begin{document}
\maketitle

\begin{abstract}
If $\Zv \in \re^{V(X)}$ is an eigenvector for eigenvalue $\lambda$ of a graph $X$ and $\alpha$ is an automorphism of $X$, then $\alpha(\Zv)$ is also an eigenvector for $\lambda$. Thus it is rather exceptional for an eigenvalue of a vertex-transitive graph to be simple. We study cubic vertex-transitive graphs with a non-trivial simple eigenvalue, and discover remarkable connections to arc-transitivity, regular maps and Chebyshev polynomials.
\end{abstract}

%

\section{Introduction}

There is a surprising inverse correlation between the number of distinct eigenvalues of a graph and the size of its automorphism group. If the automorphism group of a graph $G$ is arc-transitive, the graph has at most two simple eigenvalues. Conversely, if a connected graph on $n$ vertices has at most 2 distinct eigenvalues, then the graph is complete and the automorphism group is the full symmetric group of $n$ elements. We would like to study classes of graphs with many automorphisms and several simple eigenvalues, with the intuition that they should not be many in number and with the hope that we may describe them.

We consider a vertex-transitive graph $X$ on $n$ vertices and let $A$ denote its adjacency
matrix. We speak of eigenvalues and eigenvectors of $A$ and of $X$ interchangeably. There has been extensive study about the interplay of eigenvalues of a graph and various graph properties, such as the diameter \cite{C89, M91} or the chromatic number \cite{H70, H95}; see also \cite{MoPo93}. The relationship between symmetries of a graph and its eigenvalues has also been investigated extensively, for example in \cite{CG97, R07, R13}.

In this paper, we focus on simple eigenvalues of cubic vertex-transitive graphs. If $\lambda$ is a simple eigenvalue of such a graph, it must be equal to $\pm 3$ or to $\pm 1$. Cubic vertex-transitive graphs have been studied extensively \cite{PoSpVe12, PoSpVe10} and a census of all such graphs with at most 1280 vertices is maintained by Poto\v{c}nik, Spiga and Verret in \cite{census3}.

Using tools from diverse areas including topological graph theory, number theory and Chebyshev polynomials,  we study the combinatorial structure of cubic, vertex-transitive graphs with $\lambda=1$ as a simple eigenvalue and give several families of graphs with such spectral property, and completely classify some of special subfamilies. Somewhat more generally, we classify all generalized Petersen graphs (which have one or two orbits under the automorphism group action) with simple eigenvalue 1. We also consider the possibility that $-1$ and $+1$ are both simple eigenvalues, and we prove that this happens only when the graph is bipartite.

\begin{figure}[htbp]
  \centering
  \includegraphics[scale=0.9]{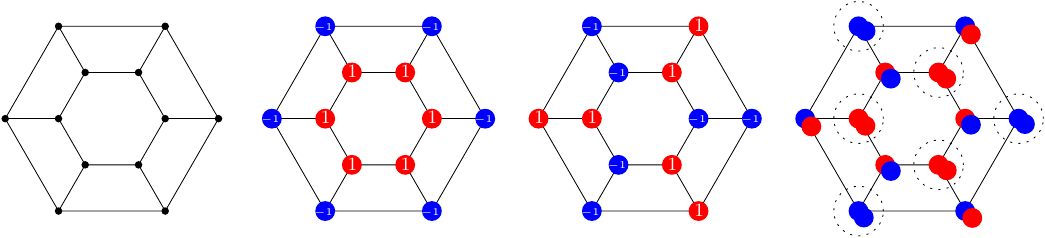}
  \caption{Eigenvectors for the prism graph of order $12$. \textit{Left to right:} the prism graph on $12$ vertices, an eigenvector with eigenvalue $1$, an eigenvector with eigenvalue $-1$, and the bipartite classes of the prism with classes determined by whether or not the eigenvectors agree or disagree at a vertex. \label{fig:prism6}}
\end{figure}

For example, there are $85$ connected cubic graphs, up to isomorphism, on $12$ vertices; of these $21$ have $1$ as a simple eigenvalue. There is exactly one graph (up to isomorphism) on $12$ vertices,  which is vertex-transitive, cubic and has $1$ as a simple eigenvalue, which is the prism graph on $12$ vertices, as shown in Figure \ref{fig:prism6}. This graph has both $1$ and $-1$ as simple eigenvalues. The eigenvectors, depicted as assignments $\pm 1$ to the vertices, are shown in Figure \ref{fig:prism6}; we will see in Section \ref{sec:partns-eigs}, from a classical result of Petersdorf and Sachs \cite{PS70}, that any vertex-transitive graph with an eigenvector with entries in $\pm 1$. Here, colouring a vertex with colour $0$ if the eigenvectors for $1$ and $-1$ agree and with colour $1$ otherwise results in a proper $2$-colouring of the graph; one of the colour classes is shown by a dotted circle in the rightmost picture in  Figure \ref{fig:prism6}. We show this holds in general, in Section \ref{sec:bip}.

The organization of the paper is as follows.
In Section \ref{sec:partns-eigs}, we give  preliminaries regarding eigenvectors of simple eigenvalues.
In Section \ref{sec:bip}, we use these to show the previously mentioned result that having both $1,-1$ as simple eigenvalues implies that a cubic graph is bipartite. We continue to extract more information about the structure of the graph, as constrained by the eigenvector, in Section \ref{sec:comb}.
In Section \ref{sec:regmaps}, we find a connection to regular maps; we show that the vertex deletion of a regular map gives a cubic vertex-transitive graph with $1$ as a (not necessarily simple) eigenvalue.
Finally, in Section \ref{sec:families},  we give several infinite families of examples of cubic vertex-transitive graphs with $1$ as a simple eigenvalue; in each case, we find an infinite family of cubic vertex-transitive graph with $1$ as an eigenvalue and then we classify when $1$ occurs as a simple eigenvalue.  In particular, we classify which generalized Petersen graphs have $1$ as a simple eigenvalue, using classical results in number theory about vanishing roots of unity and sums of cosines.

\section{Partitions and eigenvectors}\label{sec:partns-eigs}

Consider a vertex-transitive graph $X$ with vertex-set $V$ and let $A$ denote its adjacency matrix. We index rows and columns of the adjacency matrix $A = A(X)$  of $X$ by vertices of $X$. We will use functional notation where $A(x,y)$ denotes the $(x,y)$-entry of $A$. For a vector $\Zv$ indexed by the vertices of $X$ and a vertex $x$ of $X$, we write $\Zv(x)$ for the entry of $\Zv$ corresponding to $x$.

An eigenvalue of a graph $X$ is \emph{simple} if the corresponding eigenspace is 1-dimensional. Let $\lambda$ be an
eigenvalue of $X$ with eigenvector $\Zv$. The elements of the automorphism group of $X$, denoted $\Aut(X)$,
can be represented by $V\times V$ permutation matrices $P$ such that $P^T A P = A$.
Note that $P^T = P^{-1}$.

We may observe that if $A \Zv = \lambda \Zv$ and $\Zv \neq 0$, then
\[
A (P\Zv) = \lambda P \Zv
\]
and thus $P\Zv$ is also an eigenvector of $A$ with eigenvalue $\lambda$. Therefore, any automorphism $P$ of $X$ fixes the eigenspaces of $A$.

In particular, if $\lambda$ is a simple eigenvalue of $X$, then, the eigenspace of $\lambda$ has
dimension 1 and so
\[
P\Zv = \gamma \Zv
\]
for some scalar $\gamma \in \re$. Since $P$ is a permutation matrix, we have
$\gamma \in \{1, -1\}$. If $P$ is an automorphism of $X$ mapping
vertex $x$ to $y$, then $\Zv(x) = \pm \Zv(y)$.
Since $X$ is vertex-transitive, for each pair of vertices $x,y$, there exists an automorphism $P$ mapping
$x$ to $y$. Therefore $\Zv$ has entries $\pm \beta$ for some $\beta \in \re$. We may
scale the eigenvector to obtain that $\Zv$ is a $\pm 1$ vector.

We have the following standard theorem, which can be found in \cite{Bi93} or \cite{CDS95}.

\begin{theorem}[Petersdorf and Sachs \cite{PS70}]
\label{thm:PS70}
Let $X$ be a vertex-transitive graph of degree $k$.
If $\lambda$ is a simple eigenvalue of $X$, then
\[ \lambda = k - 2\alpha \]
for some integer $\alpha \in \{0, \ldots, k\}$.
\end{theorem}

\begin{proof}
Let $\lambda$ be a simple eigenvalue of $X$ and let $\Zv$ be its $\pm 1$ eigenvector.
Let $x$ be a vertex of $X$. Without loss of generality, we may assume $\Zv(x) = 1$.
We have that
\begin{equation}\label{eq:ps}
\sum_{y\sim x} \Zv(y) = \lambda \Zv(x) = \lambda.
\end{equation}
Let $\alpha$ $(0\leq \alpha \leq k)$ be the number of neighbours $y$ of $x$ such that $\Zv(y) = -1$. Then (\ref{eq:ps}) implies that $\lambda = k-2\alpha$.
\end{proof}

This proof shows that $X$ has a $\pm 1$ eigenvector whose signs determine a partition
$$
   V(X) = V^+ \cup V^-
$$
such that the induced subgraphs $X[V^+]$ and $X[V^-]$ are $(k-\alpha)$-regular and the bipartite subgraph between $V^+$ and $V^-$ is $\alpha$-regular. Conversely, every such partition determines a $\pm 1$ eigenvector of $X$ for eigenvalue $\lambda = k-2\alpha$. We have the following observation.

\begin{proposition}
\label{prop:pm1vector}
Let $X$ be a connected $k$-regular graph with an eigenvector $\Zv$ for an eigenvalue $\lambda$, whose coordinates are all $\pm1$. Then $\lambda$ is an integer and $\lambda \equiv k \pmod 2$. The sets $V^+=\{x\in V(X)\mid \Zv(x)=1\}$ and $V^- = V(X)\setminus V^+$ induce $(\frac{k+\lambda}{2})$-regular subgraphs, while the edges joining $V^+$ and $V^-$ form a $(\frac{k-\lambda}{2})$-regular bipartite subgraph of $X$. Conversely, every such partition determines a $\pm1$ eigenvector for $\lambda$.
\end{proposition}

\begin{proof}
If $V^+=V(X)$ or $V^-=V(X)$, then $\lambda = k$ and there is nothing to prove. Otherwise, eq. (\ref{eq:ps}) gives the rest of the claims.
\end{proof}

\section{Cubic vertex-transitive graphs having $1$ and $-1$ as simple eigenvalues}\label{sec:bip}

A cubic graph $X$ has largest eigenvalue equal to $3$, which is simple if and only if $X$ is connected.
It is well-known that if $-3$ is also an eigenvalue of $X$ and $X$ is connected, then $-3$ is a simple
eigenvalue and $X$ is bipartite. By Theorem \ref{thm:PS70}, the only possible simple eigenvalues of a cubic vertex-transitive graph besides $\pm 3 $ are $\pm 1$.

A partition $\{V_1, \ldots, V_m\}$ of the vertices of a graph $X$ is said to be \emph{equitable} if the subgraph of $X$ induced by each $V_i$ is regular and the bipartite subgraph of $X$ induced by the edges from $V_i$ to $V_j$ is semi-regular, for each pair $i,j$ such that $i \neq j$. If that is the case, then we define the $m\times m$ \emph{quotient matrix} $B = [b_{ij}]_{i,j = 1}^m$ whose entries $b_{ij}$ are number of neighbours of any vertex in $V_i$ in $V_j$.

\begin{theorem} If a cubic vertex-transitive graph $X$ has both $1$ and $-1$ as simple eigenvalues, then $X$ is bipartite.
\end{theorem}

\begin{proof}
Let $\Zv$ and $\Zu$ be the $\pm 1$ eigenvectors for eigenvalues $1$ and $-1$, respectively. Let
\[
\begin{split}
V^+ &= \{x \in V(X) \mid \Zv(x) = 1\}, \\
V^- &= \{x \in V(X) \mid \Zv(x) = -1\}, \\
U^+ &= \{x \in V(X) \mid \Zu(x) = 1\}, \hbox{and} \\
U^- &= \{x \in V(X) \mid \Zu(x) = -1\}.
\end{split}
\]
For any automorphism $P$ in $\Aut(X)$, we have that $P$ must either fix both $V^+$ and $V^-$
or interchange them as sets. Similarly, $P$ either fixes both $U^+$ and $U^-$
or interchanges them. By using (\ref{eq:ps}), we see that $V^+$ and $V^-$ each induce a 2-regular subgraph
of $X$ and $U^+$ and $U^-$ each induce a 1-regular subgraph of $X$.

Let $W^{++} = V^+ \cap U^+ $, $W^{+-} = V^+ \cap U^-$, $W^{-+} = V^- \cap U^+ $
and $W^{--} = V^- \cap U^- $. Consider the subgraph $Y$ induced by the vertices in $W^{++}$. Since $W^{++}\subseteq U^+$,
each vertex of $W^{++}$ has degree 0 or 1 in $Y$. Since $X$ is vertex-transitive, the
automorphism group of $X$ must also act transitively on $Y$. Then $Y$ is either 1-regular (an induced matching)
or an independent set of vertices. The same conclusion applies to $W^{+-}$, $W^{-+}$, and $W^{--}$.

If $Y$ is 1-regular, then we easily conclude that the quotient matrix of the partition of $V(X)$ induced
by $W^{++}, W^{+-}, W^{-+}, W^{--}$ must be
\[
B = \pmat{ 1 & 1 & 0 &1 \\ 1 & 1 & 1 & 0\\
0 & 1 & 1 & 1\\ 1 & 0 & 1 & 1}
\]
and this partition is equitable. By the interlacing theorem (see \cite[Theorem 2.5.1]{BH}, for example), the eigenvalues of $B$ are a sub-multiset of the eigenvalues of $A$. The matrix $B$ has eigenvalue $1$ with multiplicity 2 and so $A(X)$ also has eigenvalue 1 with multiplicity at least 2. This contradicts
the assumption that 1 is a simple eigenvalue of $X$.

Therefore, it must be that $W^{++}$ is an independent set. In this case, by vertex transitivity, the same holds for
$ W^{+-}, W^{-+}$ and $W^{--}$. This implies that each vertex in $W^{++}$ has two neighbours in $W^{+-}$, one neighbour in $W^{-+}$, and no neighbours in $W^{++} \cup W^{--}$. In particular, the partition of $V(X)$ into sets $W^{++} \cup W^{--}$ and $W^{+-} \cup W^{-+}$
is a bipartition of the graph~$X$.
\end{proof}

\section{Combinatorial structure} \label{sec:comb}

We now consider a cubic vertex-transitive graph $X$ that has $\lambda = 1$ as a simple eigenvalue with eigenvector $\Zv$ whose entries are in $\{1, -1\}$. We define vertex sets $V^+$ and $V^-$ as in the previous section. In this section, we will extract more information about the combinatorial structure of $V^+$ and $V^-$ in the graph.

For $W \subseteq V(X)$, we use $X[W]$ to denote the subgraph of $X$ induced by $W$. Let $M$ denote the set of edges between $V^+$ and $V^-$; that is
\[
M =  \{e \in E(X) \, \mid \, e = xy \text{ for } x \in V^+ \text{ and } y \in V^-  \}.
\]

\begin{lemma}
\label{lem:gplusminus}
For $(V^+, V^-)$ and $M$ as defined above, the following statements are true:
\begin{enumerate}[\rm (i)]
\item $X[V^+]$ is the disjoint union of cycles of the same length;
\item $X[V^+]$ is isomorphic to $X[V^-]$ and $V^+$ and $V^-$ are blocks of imprimitivity of the action of $\Aut(X)$;
\item $\bigl\{V^+, V^-\bigr\}$ is the unique partition of $V(X)$, such that both parts induce 2-regular graphs;
\item $M$ is a perfect matching of $X$; and
\item $\Aut(X)$ acts arc-transitively on $M$ and fixes $M$ set-wise.
\end{enumerate}
\end{lemma}

\begin{proof}
For every vertex $x\in V^+$, we have that $\sum_{y\sim x} \Zv(y) = \Zv(x) = 1$.
Since $\Zv(y)$ for all $y$ neighbours of $x$ are either $1$ or $-1$, it follows that $x$ is adjacent to two vertices in $V^+$ and one vertex in $V^-$. This implies that $M$ is a perfect matching of $X$ and $X[V^+]$ is a $2$-regular graph.

Any partition of $V(X)$ into sets $(V_1, V_2)$ such that the induced graphs $X[V_1]$ and $X[V_2]$ are $2$-regular gives rise to an eigenvector for $X$ with eigenvalue $1$, by taking the vector $\Zu$ defined as follows:
\[
\Zu(v) =
\begin{cases} ~1, & \text{if } v \in V_1, \\
-1, & \text{if } v \in V_2. \end{cases}
\]
Since $1$ is a simple eigenvalue of $X$, it follows that $\{V^-, V^+\}$ is the only such partition. Then every automorphism of $X$ must fix $V^+$ and $V^-$ or must swap $V^+$ and $V^-$ set-wise. This shows (v). Observe that there is an automorphism of $X$ taking a vertex of $V^+$ to a vertex in $V^-$. Such an automorphism must take every vertex in $V^+$ to a vertex in $V^-$ and every vertex in $V^-$ to a vertex in $V^+$ and so is an isomorphism from $X[V^+]$ to $X[V^-]$. This shows that (ii) holds. Since $\Aut(X)$ acts transitively on $X$, the induced action on $V^+$ is also transitive, so $X[V^+]$ is a vertex{sec:regmaps}-transitive $2$-regular graph. Then $X[V^+]$ must be a disjoint union of cycles of the same length.
\end{proof}

Lemma \ref{lem:gplusminus} motivates the question to classify cubic vertex-transitive graphs that admit a decomposition into a ``bipartite" 2-factor and a perfect matching, where both factors are invariant under the full automorphism group. Inspired by this problem, Alspach, Khodadadpour, and Kreher \cite{AlsKhoKre29} classified all cubic vertex-transitive graphs containing a Hamilton cycle that is invariant under the action of the automorphism group. In Section \ref{sec:families}, we classify the cases when $G[V^+]$ is a single cycle and when the cycles in $G[V^+]$ are triangles, respectively.

We may assume that each of $X[V^+]$ and $X[V^-]$ is the disjoint union of $m$ cycles of length $k$.
In this case we say that $X$ is of \emph{type $C(m,k)$}.

Let $C_i$ and $D_i$ $(i = 1, \ldots, m)$ be the cycles forming $X[V^+]$ and $X[V^-]$, respectively, as observed above. Let $G$ be the multigraph obtained from $X$ by contracting each cycle $C_i$ and each cycle $D_i$ to a single vertex. More precisely, $G$ has $2m$ vertices $c_1,\dots,c_m$ and $d_1,\dots,d_m$ (one for each cycle $C_i$ or $D_i$); there is an edge joining $c_i$ and $d_r$ in $G$ for each edge of $X$ joining a vertex in $C_i$ to a vertex in $D_r$. We say that $G$ is the \emph{contracted multigraph} of $X$. Observe that $G$ is a $k$-regular connected graph.

\begin{lemma}\label{lem:emb-auts}
If $G$ is the contracted multigraph of $X$, then $\Aut(X) \leq \Aut(G)$ and any vertex-transitive subgroup of $\Aut(X)$ acts transitively on the arcs of $G$. In particular, $G$ is arc-transitive and bipartite.
\end{lemma}

\begin{proof}
Consider any $\alpha \in \Aut(X)$. If we take the contracted multigraph of $\alpha(X)$, we again obtain $G$, since the cycles $C_i, D_j$ form blocks of imprimitivity under the action of $\Aut(X)$. Thus, $\alpha$ acts on $G$ as an automorphism, and so $\Aut(X) \leq \Aut(G)$. Let $\Gamma \leq \Aut(X)$ act transitively on the vertices of $X$. Then $\Gamma$ acts arc-transitively on the edges of $M$, which are in one-to-one correspondence with the edges of $G$, and thus acts arc-transitively on $G$. Clearly every edge in $G$ connects some $c_i$ to some $d_j$, so $G$ is bipartite.
\end{proof}

\section{Relation to regular maps}\label{sec:regmaps}

In this section, we explore the relation between the combinatorial structure given by the eigenvector of $1$, as a simple eigenvalue, in a cubic vertex-transitive graph and the combinatorial structure of graph embeddings with a high degree of symmetry. In this setting, we can obtain from a regular map a cubic vertex-transitive graph with $1$ as an eigenvalue.

First, we proceed with some preliminary definitions from the area of graph embeddings. Further details may be found in \cite{MT01}. Let $G$ be a connected multigraph. For each $v \in V(G)$, let $\pi_v$ be a cyclic permutation of the edges incident to $v$. Then $\Pi = \{\pi_v \mid v\in V(G)\}$ is said to be a \emph{rotation system} for $G$. An automorphism $\alpha$ of $G$ is said to \emph{preserve the rotation system} if it maps the local rotation around each vertex $v$ onto the local rotation of $\alpha(v)$. The subgroup of $\Aut(X)$ that preserves the rotation system $\Pi$ is called the \emph{automorphism group of the graph with rotation system}, and is denoted by $\Aut(G,\Pi)$.

Each rotation system describes an \emph{embedding} of $G$ on an orientable surface: it defines a collection of closed walks, called \emph{facial walks} or \emph{faces}, such that each edge is traversed once in each direction by these walks, and by pasting discs onto each facial walk, we obtain a \emph{map}, i.e. an orientable surface in which $G$ is 2-cell embedded. Thus we view the pair $(G,\Pi)$ as a map, and we call $\Aut(G,\Pi)$ the group of \emph{map automorphisms} corresponding to the map determined by the rotation system $\Pi$.
A map is said to be \emph{orientably regular} (or \emph{rotary} \cite{Wilson}) if $\Aut(G,\Pi)$ acts transitively on the arcs of $G$.

Let $X$ be a cubic graph of type $C(m,k)$.
Suppose that $\Gamma$ is a subgroup of $\Aut(X)$ that acts transitively on $V(X)$. Recall that $V^+$ and $V^-$ are blocks of imprimitivity of $\Aut(X)$ and hence also of $\Gamma$. Similarly, the cycles $C_1,\dots,C_m$ and $D_1,\dots,D_m$ form a system of blocks of imprimitivity. The stabilizer of $C_1$ in $\Gamma$ (the subgroup of all elements of $\Gamma$ that fix $C_1$) acts on the cycle $C_1$ either as a cyclic group or as a dihedral group. In this section we shall assume that the action is regular:

\begin{itemize}
\item[\bf (A1)]
The stabilizer of $C_1$ in $\Gamma$ acts regularly on $C_1$.
\end{itemize}

Note that {\bf (A1)} implies that the stabilizer of $C_1$ in $\Gamma$ preserves the orientation of $C_1$ and acts on $C_1$ regularly as the cyclic group $\ints_k$.

Let $C_1=v_1v_2\dots v_k$. Suppose that $D_1=w_1w_2\dots w_k$ is chosen so that $v_1w_1\in E(X)$.
For $i=1,\dots, m$, let $\gamma_i\in \Gamma$ be an automorphism that maps $C_1$ to $C_i$, and let $\delta_i\in \Gamma$ be an automorphism that maps $C_1$ to $D_i$. Moreover, we may assume that $\gamma_1$ rotates $C_1$ clockwise by one vertex, i.e. $\gamma_0(v_j)=v_{j+1}$ for $j=1,\dots,k$ (indices taken modulo $k$) and that $\delta_1$ maps $v_j$ to $w_j$ for $j=1,\dots,k$.

\begin{theorem}
\label{thm:regular maps 1}
If\/ $\Gamma$ satisfies {\bf (A1)}, then $\Gamma$ acts regularly on $V(X)$ and hence $X$ is a Cayley graph of the group generated by $\gamma_1$ and the involution $\delta_1$.
\end{theorem}

\begin{proof}
The group $\Gamma$ acts transitively on $V(X)$. To see that it is a Cayley graph it suffices to show that its action is regular (no fixed points). So, suppose that $\gamma\in\Gamma$ fixes a vertex $v$. Let $\alpha_v\in\Gamma$ be a group element that maps $v_1$ to $v$. Then $\alpha_v^{-1}\gamma\alpha_v$ fixes $v_1$ and by {\bf (A1)}, it must be the identity automorphism. This implies that $\gamma$ is the identity. This conclusion confirms the claim.
\end{proof}

\begin{theorem}
\label{thm:regular maps 2}
If\/ $\Gamma$ satisfies {\bf (A1)}, then the contracted multigraph $G$ of $X$ admits a rotation system $\Pi$ such that
$\Gamma \le \Aut(G,\Pi)$. The group $\Gamma$ acts arc-transitively on $(G,\Pi)$ and therefore $(G,\Pi)$ is an orientably regular map.
\end{theorem}

\begin{proof}
Fix an orientation of $C_1$ and orient each $C_i$ and $D_i$ according to the orientation induced by $\gamma_i(C_1)$ and $\delta_i(C_1)$, respectively. We claim that for each $\gamma\in\Gamma$, the orientation of the cycle $\gamma(C_1)$ is the same as the one defined above. Let us give the argument for the case when $\gamma(C_1) = D_i = \delta_i(C_1)$. In that case, $\delta_i^{-1} \gamma(C_1)$ fixes $C_1$ and hence by {\bf (A1)} fixes the orientation of $C_1$. This implies that $\gamma$ and $\delta_i$ must induce the same orientation on $D_i$, which we were to prove.

The orientations of cycles determine a rotation system on the contracted multigraph $G$ of $X$, and as shown above, $\Gamma$ preserves the rotation around the vertex corresponding to $C_1$. Suppose that there is $\gamma\in\Gamma$ that does not preserve one of the rotations, say it maps the rotation around $C_i$ onto the opposite rotation around $D_j$. (The proof of other cases is similar.) In that case, $\gamma \gamma_i$ maps $C_1$ onto $D_j$ with the opposite rotation as $\delta_j$, which contradicts what we have proved above. This completes the proof.
\end{proof}

A special case when $m=1$ gives rise to regular embeddings of the 2-vertex contracted multigraph (with $k$ parallel edges). This case will be treated in a somewhat greater generality in Section \ref{sect:genPetersen}.

\begin{figure}[htbp]
  \centering
  \includegraphics[scale=1]{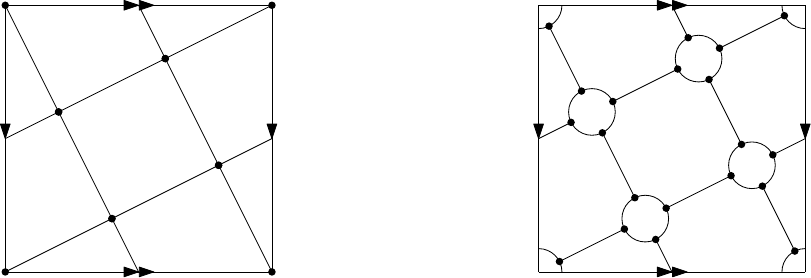}
  \caption{An example of a regular map and its vertex truncation (on the right).   \label{fig:k5}}
\end{figure}

Given a graph embedding $(G,\Pi)$, we define the \textsl{vertex truncation} of  $(G,\Pi)$, denoted $T(G,\Pi)$ as follows: the vertices of $T(G,\Pi)$ are all incident pairs $(v,e)$ for $v \in V(G)$ and $e \in E(G)$ such that $v$ is incident to $e$. Two vertices of $T(G,\Pi)$, say $(v,e)$ and $(w,f)$ are adjacent if $e=f$ or if $v = w$ and $\pi_v(e) = f$. Roughly speaking, we obtain $T(G,\Pi)$ from $G$ by replacing each vertex of $G$ with a cycle determined by $\Pi$. Figure \ref{fig:k5} shows an example a vertex truncatio; on the left side of the figure, we have the complete graph $K_5$ embedded in the torus as a regular map and, on the right, we have the vertex truncation of this embedding.

\begin{lemma} If $G$ is a bipartite, arc-transitive graph and $(G, \Pi)$ is an embedding of $G$ on some orientable surface, then the vertex truncation   $T(G,\Pi)$ is a vertex-transitive cubic which has $1$ as an eigenvalue.  \end{lemma}

\begin{proof}
Each vertex $(v,e)$ of  $T(G,\Pi)$ is adjacent to exactly three vertices: $(v, \pi_v(e))$, $(v, \pi_v^{-1}(e))$ and $(w, e)$, where $e =vw$. Since $G$ is an arc-transitive graph, the automorphism group also acts transtively on the vertices of $T(G, \Pi)$, preserving adjacencies in $T(G, \Pi)$. Thus $T(G, \Pi)$ is a cubic vertex-transitive graph.

Let $(A,B)$ be the bipartition of $G$. We will partition the vertices of $T(G, \Pi)$ into $A' \cup B'$ as follows: let
\[ A' = \{ (v,e) \,\mid\, v \in A \} \text{ and } B' = \{ (u,f) \,\mid\, u \in B \}.
\]
We note that each vertex $(v,e=vw)$  of $A'$ has two neighbours, $(v, \pi_v(e))$ and $(v, \pi_v^{-1}(e))$, in $A'$ and one neighbour $(w, e)$ in $B'$. Similarly, each vertex of $B'$ has two neighbours in $B'$ and one neighbour in $A'$. Thus the vector which takes value $1$ for each vertex of $A'$ and value $-1$ for each vertex of $B'$ will be an eigenvector for $A(T(G,\Pi))$ with eigenvalue $1$.
\end{proof}

We note that the vertex truncation of a bipartite arc-transitive graph does not necessarily have $1$ as a simple eigenvalue, but such graphs are good candidates for having $1$ as a simple eigenvalue because they have a $\pm1$ eigenvector for eigenvalue $1$, like the vertex-transitive cubic graphs with $1$ as a simple eigenvalue. The example in Figure \ref{fig:k5} does not have $1$ as a simple eigenvalue. We note that if a graph embedding $(G,\Pi)$ is a regular map, then $G$ is arc-transitive.

In Appendix \ref{sec:app-computation}, we construct the vertex truncations of small regular maps from the census of Marston Conder \cite{Con94, Conders} and check when $1$ is indeed a simple eigenvalue. For example, the M\"{o}bius-Kantor graph has an embedding in the double torus which is a regular map; this embedding has with six octagonal faces and the full automorphism group, whose order is $96$, of the graph can be realized as the automorphism group of the map \cite{Co50}. We computed that the vertex truncation of this embedding of the M\"{o}bius-Kantor graph has $1$ as a simple eigenvalue, see Table \ref{tab:regmaps}.

\section{Families of graphs}
\label{sec:families}

In this section we show the existence of several infinite families of cubic vertex-transitive graphs with $1$ as a simple eigenvalue. While it is not difficult to find infinite families of such graphs with $1$ as an eigenvalue, it is often difficult to determine that $1$ is a simple eigenvalue. For these family, we have the characterization of when $1$ is simple, using some number theoretic methods, as well as results about the sum of cosine, resulting from vanishing sums of roots of unity.

\subsection{Cubic multigraphs (type $C(m,2)$)}

Let us first discuss the case when $X$ contains multiple edges. When there is a triple edge, we have a two-vertex graph, denoted $K_2^3$ in Figure \ref{fig:F2n cover}, whose eigenvalues are $\pm 3$. Otherwise, there are only double edges and single edges. Since $X$ is vertex-transitive, every vertex $x$ has two neighbors, one is joined to $x$ by a double edge, the other one by a single edge.  It follows that $X$ is obtained from an even cycle $C_{2n}$ by adding an edge in parallel to every second edge on the cycle.
Let us denote this graph by $F_{2n}$.

\begin{figure}[htbp]
  \centering
  \includegraphics[width=8cm]{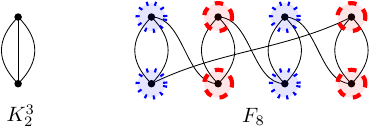}
  \caption{$F_{8}$, \textit{left}, is a regular cyclic cover of $K_2^3$, \textit{right}; the top four vertices of $F_8$ forms the fibre corresponding to the top vertex of $K_2^3$ and, likewise, the bottom vertices of $F_8$ forms the fibre corresponding to the bottom vertex of $K_2^3$. The partition given by the blue and red decorations on the vertices is the partition into $V^+$ and $V^-$}
  \label{fig:F2n cover}
\end{figure}

This graph has unique partition into sets $V^+$ and $V^-$ as in Lemma \ref{lem:gplusminus}; the induced graphs on $V^+$ and $V^-$ consist of $m$ digons each, thus $n$ is even and $X = F_{4m}$. Figure \ref{fig:F2n cover} shows this partition for $F_8$. Therefore, it has a unique $\pm1$ eigenvector for eigenvalue $\lambda=1$. It is a simple exercise to exclude eigenvectors for $\lambda=1$ that are not multiples of this one. Instead of doing this, we note that $F_{2n}$ is a regular cyclic cover over the 2-vertex graph with triple edge, $K_2^3$ (see Figure \ref{fig:F2n cover}). Therefore its eigenvalues are (see \cite{KwakLee92} or \cite{MoTR15} for details) the union of eigenvalues of $n$ matrices of the form
$$
    A_j = \pmat{0 & 2+\exp(\tfrac{2\pi {\rm i} j}{n}) \\ 2+\exp(\tfrac{-2\pi {\rm i} j}{n}) & 0}, \qquad j=0,1,\dots,n-1.
$$
The matrix $A_j$ has eigenvalue 1 if and only if $n$ is even and $j=n/2$. This gives the following result.

\begin{proposition}
The eigenvalues of the graph $F_{2n}$ are $$\pm\sqrt{4 \cos(2\pi j/n) + 5}, \quad j=0,1,\dots,n-1.$$
Thus $\lambda=1$ is an eigenvalue if and only if $n$ is even, in which case this is a simple eigenvalue.
\end{proposition}

\begin{proof}
A short computation shows that the eigenvalues of the matrix $A_j$ are $$\pm\sqrt{4 \cos(2\pi j/n) + 5}.$$
Such an eigenvalue is equal to 1 if and only if $\cos(2\pi j/n) = -1$, i.e. $j=n/2$.
\end{proof}

\subsection{Truncations of cubic arc-transitive graphs (type $C(m,3)$)}
\label{subsec:truncation}

The \emph{truncation} of a cubic multigraph $G$ is a cubic graph $T(G)$ where every vertex $v$ of $G$ corresponds to a triangle and every edge $uv$ of $G$ gives an edge of $T(G)$ between the triangles corresponding to $u$ and $v$. Figure \ref{fig:truncs} shows two examples of truncations of cubic graphs. The truncation is isomorphic to the line graph of the subdivision of $G$. Thus, the eigenvalues of $T(G)$ are easy to compute from the eigenvalues of $G$ (see \cite{CDS95} or \cite[Theorem 2.1]{Zh09}).

\begin{figure}[htbp]
  \centering
  \includegraphics[scale=0.75]{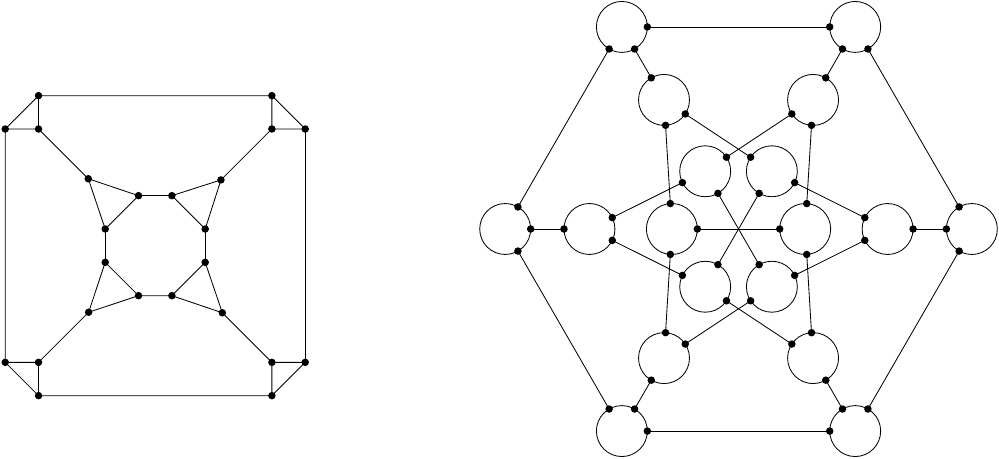}
  \caption{The truncation of the cube graph, on the left, and the truncation of the Pappus graph, on the right.   \label{fig:truncs}}
\end{figure}

\begin{theorem}\label{thm:cantfind}
If the eigenvalues of a cubic graph $G$ are $\mu_1, \ldots, \mu_n$, the eigenvalues of the truncation of $G$ are
\[
\lambda_i = \frac{1 \pm \sqrt{4\mu_i + 13}}{2}
\]
for $i = 1,\ldots, n$ and $-2$ and $0$, each with multiplicity $\frac{n}{2}$.
\end{theorem}

This implies the following result.

\begin{corollary} \label{cor:truncation}
For a connected cubic graph $G$ the following statements are equivalent:
\begin{itemize}
\item[\rm (a)]
$G$ is bipartite.
\item[\rm (b)]
$\lambda=1$ is an eigenvalue of the truncation $T(G)$ of $G$.
\item[\rm (c)]
$\lambda=1$ is a simple eigenvalue of $T(G)$.
\end{itemize}
\end{corollary}

\begin{proof}
We see from Theorem \ref{thm:cantfind} that $\lambda_i=1$ if and only if $\mu_i=-3$. A connected cubic graph has $-3$ as an eigenvalue if and only if it is bipartite, in which case $-3$ is a simple eigenvalue (and hence $1$ is a simple eigenvalue of $T(G)$).
\end{proof}

\begin{corollary}
 If $X$ is a connected vertex-transitive cubic graph containing a cycle of length 3, then $X$ has 1 as a simple eigenvalue if and only if $X = T(G)$, where $G$ is a connected, bipartite, arc-transitive cubic multigraph.
\end{corollary}

\begin{proof} First we show that if $X$ is vertex-transitive, contains a triangle and has $1$ as a simple eigenvalue, then $X$ must be the truncation of a graph $G$. Since $X$ has $1$ as a simple eigenvalue, we may partition the vertices of $X$ into cycles $V^+$ and $V^-$, as in Lemma \ref{lem:gplusminus}. Let $v$ be vertex in $V^+$. We have that $v$ is incident to a cycle of length $3$, say $T$, in $X$. Since there is a matching between $V^+$ and $V^-$, the triangle $T$ does not use any edge of the matching. Thus $v$ must be incident to a cycle of length 3 in the subgraph of $X$ induced by $V^+$. Since $V^+$ induces a vertex-transitive $2$-regular subgraph of $X$ by Lemma \ref{lem:gplusminus}, $X[V^+]$ must be a disjoint union of cycles of length $3$. The same holds for $X[V^-]$. Let $G$ be obtained from $X$ by contracting all the $3$-cycles in $X[V^+]$ and $X[V^-]$. We see that $G$ is cubic, since each vertex of $X$ is incident to exactly $1$ edge which is not contracted to obtain $G$. By part (v) of Lemma \ref{lem:gplusminus}, $G$ is an arc-transitive bipartite graph, as claimed.

The converse implication is clear by Corollary \ref{cor:truncation}.
\end{proof}

Since the cube graph and the Pappus graph are both cubic, arc-transitive, bipartite graphs, their truncations, shown in Figure \ref{fig:truncs} are examples of cubic vertex-transitive graphs with $1$ as a simple eigenvalue.

\subsection{Prisms and generalized Petersen graphs}\label{sect:genPetersen}

In this section, we classify which prisms and which generalized Petersen graphs have $1$ as a simple eigenvalue.

The \emph{prism of order $2n$} is the Cartesian product of $C_n$ with $K_2$. The prism of order $12$ appears in Figure \ref{fig:prism6}. The eigenvalues of the prism graph of order $2n$ are (see, e.g., \cite{CDS95})
$$
    2\cos \frac{2\pi j}{n} \pm 1, \qquad j=0,1,\dots,n-1.
$$
When $j=0$, this gives the eigenvalue $1$. Thus, a prism is a cubic vertex-transitive graph, which always has 1 as an eigenvalue. However, this eigenvalue is not always simple. This happens if and only if $\cos \frac{2\pi j}{n} = 0$ for some $j$; which is the case if and only if $j=n/4$ or $j=3n/4$. This immediately gives the following characterization.

\begin{proposition}
The prism of order $2n$ has $\lambda = 1$ as simple eigenvalue if and only if \linebreak $n \not\equiv 0\ (\bmod\ 4)$. \end{proposition}

Prisms are a special case of the generalized Petersen graphs, where the multiplicity of eigenvalue $1$ is relatively straightfoward to understand. We now turn our attention to the more general case.

The \emph{generalized Petersen graph}, denoted $P(n,k)$ is the graph with vertex set $[n]\times [2]$, where $[n] = \{1,\dots,n\}$, and each vertex $(j,1)$ is adjacent to $(j,2)$ and to the vertices $(j\pm1,1)$, while $(j,2)$ is adjacent to $(j\pm k,2)$, all operations taken modulo $n$. The well-known Petersen graph is isomorphic to $P(5,2)$ and the prism of order $2n$ is isomorphic to  $P(n,1)$. See Figure \ref{fig:genpete} for some other examples of generalized Petersen graphs.

\begin{figure}[htbp]
  \centering
  \includegraphics[scale=0.78]{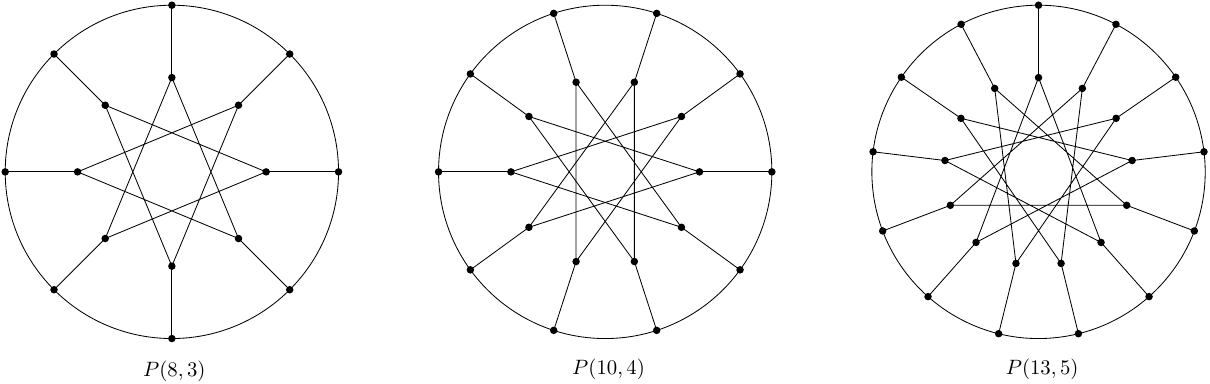}
  \caption{Examples of generalized Petersen graphs.   \label{fig:genpete}}
\end{figure}

The vertex-transitivity of the generalized Petersen graphs is given in the following theorem.

\begin{theorem}[\cite{FGW71}]\label{thm:vtpete}
The generalized Petersen graph $P(n,k)$ is vertex-transitive if and only if $(n,k) = (10,2)$ or $k^2 \equiv \pm 1 \ (\bmod\ n)$.
\end{theorem}

The following theorem gives the eigenvalues of $P(n,k)$.

\begin{theorem}[\cite{GS11}]
The graph $P(n,k)$ has eigenvalues $\delta$ for every  root  $\delta$ of
\begin{equation} \label{eq:minpoly}
   x^2 - (\alpha_j + \beta_j)x + \alpha_j \beta_j - 1 = 0
\end{equation}
for $j = 0, \ldots, n-1$, where
\[
   \alpha_j = 2\cos(\tfrac{2\pi j}{n}) \text{\quad and \quad} \beta_j  = 2\cos(\tfrac{2\pi jk}{n}).
\]
\end{theorem}

The eigenvalues of $P(n,k)$ which are equal to 1 are solutions for Equation (\ref{eq:minpoly}) where $x = 1$, which we may simplify as:
\[
\alpha_j + \beta_j = \alpha_j  \beta_j .
\]
We may let $\theta = \frac{2\pi j}{n}$ and rewrite as:

\begin{equation} \label{eq:1eqn}
 \cos\theta + \cos k\theta = 2\cos\theta \cos k\theta = \cos((k-1)\theta) + \cos((k+1)\theta).
\end{equation}
Observe that $j=0$ gives a solution to this equation for any $k$. Hence, every generalized Petersen graph has eigenvalue $1$ with multiplicity at least one.

\begin{theorem}\label{thm:genPeteClassifcation} $P(n,k)$ has $1$ as a simple eigenvalue if and only if one of the following holds:
\begin{enumerate}
  \item[\rm (i)] $4 \nmid n$ and $5 \nmid n$.
  \item[\rm (ii)] $4 \mid n$ and $k$ is even.
  \item[\rm (iii)] $5 \mid n$ and $k \notin \{ 2,3, n-3, n-2\}$.
\end{enumerate}
\end{theorem}

\begin{proof}
As shown above, every $P(n,k)$ has eigenvalue 1 corresponding to the solution $j=0$ to \eqref{eq:1eqn}. We note that solutions to \eqref{eq:1eqn} are equivalent to solution to
\eqref{eqn:roots1} of the form $(j,jk)$ and thus we will make use of the solutions that we found in the proof of Lemma \ref{lem:45cos}. In light of this, Lemma \ref{lem:cossums} shows that $1$ is a simple eigenvalue when $4 \nmid n$ and $5 \nmid n$.

Suppose now that $n = 4a$, where $a$ is an integer. If $k \equiv 1 \pmod 4$ (resp. $k \equiv 3 \pmod 4$), we see that $j=a$ (resp. $j=3a$) is a non-trivial solution to \eqref{eq:1eqn}.
Suppose now that $k$ is even. By Lemma \ref{lem:cossums}, if $4 | m$, the only non-trivial solution to \eqref{eqn:roots1} are when $\left\{ \frac{j}{m},\frac{\ell}{m} \right\} \subseteq \left\{ \frac{1}{4},\frac{3}{4} \right\}$. In the case for the generalized Petersen graphs, solutions to \eqref{eqn:roots1} where $\ell = kj$ are exactly the solution to \eqref{eq:1eqn}. We see that if $k$ is even, then  $\left\{ \frac{j}{n},\frac{jk}{n} \right\}$ is  either $\left\{ \frac{1}{4},\frac{k}{4} \right\}$  or $\left\{ \frac{3}{4},\frac{3k}{4} \right\}$, neither of which can be a subset of $\left\{ \frac{1}{4},\frac{3}{4} \right\}$ since $k$ is even. Thus, in this case, $1$ is a simple eigenvalue.

Suppose now that $5 \mid n$. By Lemma \ref{lem:cossums}, if  $m = 5a$, the only non-trivial solutions to \eqref{eqn:roots1} are $(j,\ell) \in \{a, 4a\}\times\{2a, 3a\} \cup \{2a, 3a\}  \times \{a, 4a\}$. Again, we need solutions to \eqref{eqn:roots1} where $\ell = kj \mod 5a$. We summarize the values of $j, \ell$ and $k$ that we obtain in Table~\ref{tab:jklpete5}.
\begin{table}
  \centering
  \begin{tabular}{l|l|l|l|l|l|l|l|l}
$j$ & \multicolumn{2}{l|}{$a$} & \multicolumn{2}{l|}{$4a$} & \multicolumn{2}{l|}{$2a$} & \multicolumn{2}{l}{$3a$} \\ \hline
$\ell$ &      $2a$     &   $3a$        &  $2a$         &    $3a$       &  $a$         &    $4a$       &   $a$        &   $4a$        \\ \hline
$k$ &     $2$ , $n-3$     &     $3$  , $n-2$    &   $3$  , $n-2$       &     $2$ , $n-3$     &   $3$   , $n-2$      &    $2$  , $n-3$     &    $3$    , $n-2$     &     $2$  , $n-3$    \\
  \end{tabular}
  \caption{Solution to \eqref{eqn:roots1} and \eqref{eq:1eqn} when $n = 5a$.  \label{tab:jklpete5}}
\end{table}
Thus, we get additional solution if and only if $k \in\{ 2,3, n-3, n-2\} \mod n$.
\end{proof}

For the examples in Figure \ref{fig:genpete}, we see by Theorem \ref{thm:genPeteClassifcation}, that $P(10,4)$ and $P(13,5)$ have $1$ as a simple eigenvalue while $P(8,3)$ does not. In fact, $P(13,5)$ is the smallest generalized Petersen graph with $1$ as a simple eigenvalues, which is not isomorphic to a prism.

\subsection{Regular embeddings of $K_{m,m}$}\label{subsec:regular-kmm}

In this section, we consider the truncation, $T_m$, of regular embeddings of the complete bipartite graph $K_{m,m}$. This is an application of the ideas from Section \ref{sec:regmaps}. Here, we characterise which orders of $m$ give graphs where $1$ is a simple eigenvalue, using results on vanishing sums of roots of unity and sums of cosines, which are included in Appendix \ref{app:cosines}.

For $m\ge 3$, let $T_m$ be the cubic graph of order $2m^2$ defined in the following way. The vertices of $T_m$ are $\{\pls{i}{j}, \mn{i}{j} \mid i, j \in \ints_m\}$. The edges are
\[
\{\pls{i}{j}, \pls{i}{j+1}\}, \,
\{\mn{i}{j}, \mn{i}{j+1}\}, \,
\{\pls{i}{j}, \mn{j}{i}\}
\]
for all $i, j \in \ints_m$. It is easy to see that $T_m$ is a cubic, vertex-transitive graph with 1 as an eigenvalue, not necessarily simple, by considering the eigenvector that is $+1$ on every vertex $\pls{i}{j}$ and $-1$ on every vertex $\mn{i}{j}$.
Figure \ref{fig:t3t4} shows $T_3$ and $T_4$.

\begin{figure}[htbp]
  \centering
  \includegraphics{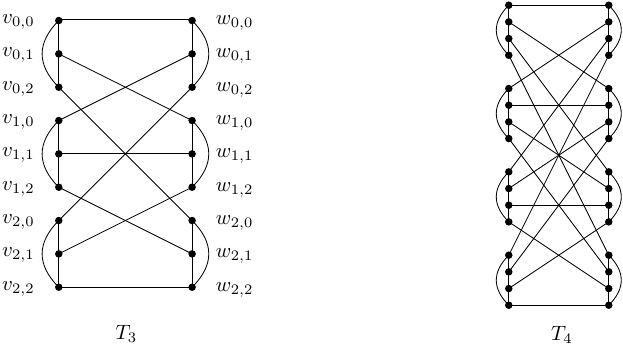}
  \caption{The graph $T_3$, on the left, and  $T_4$, on the right. The vertex labels have been repressed for $T_4$, for readability, but can be easily inferred from those for $T_3$. \label{fig:t3t4}}
\end{figure}

Alternatively, we can construct $T_m$ from the regular embedding of $K_{m,m}$, given by Nedela and \v Skoviera in \cite{NeSk97}. We consider $K_{m,m}$ as a Cayley graph on $\ints_{2m}$ with the connection set $\{1, 3, 5, \ldots, 2m-1\}$ as the generating set. The rotation system $\Pi$ has vertex rotations at each vertex given by the cyclic permutation $(1, 3, 5, \ldots, 2m-1)$ of the generators. The graph $T_m$ is isomorphic to the vertex truncation $T(K_{m,m}, \Pi)$.

Let $B$ be the $m^2 \times m^2$ matrix such that $B = I_m \otimes C_m$, where $C_m$ is the adjacency matrix of the cycle of order $m$ and $I_m$ is the $m\times m$ identity matrix. Let $P$ be the permutation matrix indexed by $\ints_m \times \ints_m$ such that $P$ takes $(i,j)$ to $(j,i)$. The adjacency matrix of $T_m$ can be written as
\[
A := A(T_m) = \pmat{B & I \\ I & P^TBP}.
\]
Observe that $P^2 = I$ and $P = P^T$. By definition, we see that
\[
P(e_i \otimes e_j) = e_j \otimes e_i
\]
where $e_k$ denotes the $k$th elementary basis vector. Then for any $m\times 1$ vectors $\Zv$ and $\Zw$, we see that
\[
P(\Zv\otimes \Zw) = \Zw \otimes \Zv.
\]

\begin{theorem}
\label{thm:Tm}
The eigenvalues of $T_m$ are
\[ \cos \tfrac{2\pi j}{m} + \cos\tfrac{2\pi\ell}{m} \pm \sqrt{ \left(\cos \tfrac{2\pi j}{m} - \cos\tfrac{2\pi\ell}{m}\right)^2 + 1}
\]
for all $(j, \ell) \in \ints_m \times \ints_m$.
\end{theorem}

\begin{proof}
To find these eigenvalues, we use the proof method developed in \cite{GS11} in order to find the eigenvalues of generalized Petersen graphs. Let $\Zv, \Zw$ be eigenvectors of $C_m$ with eigenvalues $\lambda$ and $\theta$, respectively. Then
\[
B (\Zv \otimes \Zw) = (I_m \otimes C_m)(\Zv \otimes \Zw) = \Zv \otimes \theta \Zw = \theta(\Zv \otimes \Zw)
\]
and
\[
P^TBP (\Zv \otimes \Zw) = P^T(I_m \otimes C_m)(\Zw \otimes \Zv) = P( \lambda \Zv \otimes \theta \Zw) = \lambda (\Zv \otimes \Zw).
\]
Let $V$ be an eigenbasis for $C_m$ in $\re^m$. Then the basis
\[
W = \{\Zv \otimes \Zw \mid \Zv, \Zw \in V \}
\]
of $\re^{m^2}$ simultaneously diagonalizes $B$ and $P^T BP$. We construct an eigenbasis $U$  of $A$ over $\re^{2m^2}$ such that the elements of $U$ are
\[
\pmat{\alpha \Zv \otimes \Zw \\ \Zv \otimes \Zw}
\]
where $\Zv,\Zw \in V$ with eigenvalues  $\lambda$ and $\theta$, respectively and $\alpha = \delta -\theta $ for each $\delta$ a solution to
\begin{equation}\label{eqn:eroots} \delta^2 - (\theta + \lambda) \delta + \theta \lambda -1 = 0.
\end{equation}
Observe that since
\[
\left(\frac{\theta + \lambda}{2} \right)^2 = \frac{\theta^2 + \lambda^2}{2} + \theta\lambda > \theta \lambda -1
\]
for any $\lambda, \theta \in \re$, Equation (\ref{eqn:eroots}) always has two distinct solutions for $\delta$. The set $U$ consists of $2m^2$ linearly independent vectors in $\re^{2m^2}$. We now verify that each element of $U$ is an eigenvector of $A$ by observing
\[
\begin{split}
A \pmat{\alpha \Zv \otimes \Zw \\ \Zv \otimes \Zw} &= \pmat{B & I \\ I & P^TBP} \pmat{\alpha \Zv \otimes \Zw \\ \Zv \otimes \Zw} \\
&= \pmat{\alpha \theta \Zv \otimes \Zw + \Zv \otimes \Zw \\ \alpha \Zv \otimes \Zw + \lambda \Zv \otimes \Zw} \\
&= \pmat{(\alpha \theta + 1) \Zv \otimes \Zw \\ (\alpha + \lambda)\Zv \otimes \Zw } .
\end{split}
\]
But $\alpha$ has been carefully chosen such that $\alpha + \theta = \delta$ and $\alpha\lambda + 1 = \delta \alpha$, so
\[
A \pmat{\alpha \Zv \otimes \Zw \\ \Zv \otimes \Zw} = \delta \pmat{\alpha \Zv \otimes \Zw \\ \Zv \otimes \Zw}.
\]
Then $U$ is an eigenbasis for $A$, as claimed, and the eigenvalues of $A$ are the solution for $\delta$ in Equation (\ref{eqn:eroots}) where $\theta$ and $\lambda$ range over the eigenvalues of $C_k$. We use the quadratic formula to see that
\[
    \delta = \frac{\theta + \lambda \pm \sqrt{(\theta + \lambda)^2 - 4(\theta \lambda - 1) }}{2}
           = \frac{\theta + \lambda \pm \sqrt{(\theta - \lambda)^2 + 4  }}{2}.
\]
The eigenvalues of $C_m$ are $2\cos \tfrac{2\pi j}{m}$ for $j \in \ints_m$ (see \cite{BH} or \cite{CDS95}). Putting these values in place of $\theta$ and $\lambda$, we obtain the expressions given by the theorem.
\end{proof}

By restricting our attention to the eigenvalue 1, we have the following consequence of Theorem \ref{thm:Tm}.

\begin{corollary}
\label{cor:Tm eigenvalue 1}
The multiplicity of\/ $1$ as an eigenvalue of the graph $T_{m}$ is equal to the number of solutions $(j, \ell)$ to the equation:
\begin{equation}\label{eqn:roots1}
\cos \tfrac{2\pi j}{m} + \cos \tfrac{2\pi \ell}{m} = 2 \cos \tfrac{2\pi j}{m} \cos \tfrac{2\pi \ell}{m}
\end{equation}
where $j, \ell \in \{0, \ldots m -1\}$.
\end{corollary}

\begin{proof}
Let $a = \cos\tfrac{2\pi j}{m}$ and $b = \cos\tfrac{2\pi \ell}{m}$. By Theorem \ref{thm:Tm}, we are looking for the number of pairs $(j,\ell)$ for which $a+b-1=\pm\sqrt{(a-b)^2+1}$. For $a=b$, this is satisfied precisely when $a=b=1$ and when $a=b=0$, both of which provide solutions for (\ref{eqn:roots1}).
On the other hand, if $a\ne b$, then the value under the square root is bigger than 1, which implies that the solutions must satisfy the equation $a+b-1 = -\sqrt{(a-b)^2+1}$. For values $a,b$ that are smaller or equal to 1, this is equivalent to $(a+b-1)^2 = (a-b)^2 +1$, which holds if and only (\ref{eqn:roots1}) holds.
\end{proof}

\begin{theorem} $T_m$ has $1$ as a simple eigenvalue if and only if $4 \nmid m$ and $5 \nmid m$. \end{theorem}

The theorem follows from Corollary \ref{cor:Tm eigenvalue 1} by Lemmas \ref{lem:45cos} and \ref{lem:cossums} that are proved in Appendix~\ref{app:cosines}.

\appendix

\section{Sums of cosines}\label{app:cosines}

The multiplicity of $1$ as an eigenvalue of the graph $T_{m}$ is equal to the number of solutions $(j, \ell)$ to the equation  \eqref{eqn:roots1}. We observe that, if we set $\ell=kj$, we obtain exactly \eqref{eq:1eqn} in Section \ref{sect:genPetersen}. Note that this equation is always satisfied for $j = \ell = 0$, which we will refer to as the \textsl{trivial} solution.

We will first determine that when $4 |m$ or $5|m$, equation \eqref{eqn:roots1} has non-trivial solutions. Then, we will show that there are no non-trivial solutions in the other cases.

\begin{lemma}\label{lem:45cos} If $4$ divides $m$, then  \eqref{eqn:roots1} has at least $4$ non-trivial solutions. If $5$ divides $m$, then \eqref{eqn:roots1} has at least $8$ non-trivial solutions.
\end{lemma}

\begin{proof}
We will describe additional solutions to \eqref{eqn:roots1} when $m$ is divisible by $4$ and when $m$ is divisible by $5$.
Note that
\[
\cos \tfrac{\pi}{2} = \cos\tfrac{3\pi}{2} = 0.
\]
Then, if $m = 4a$ for some $a$, then for each
\[
(j,\ell) \in \{(a,a),(a,3a),(3a,a),(3a,3a)\}
\]
we obtain a solution to \eqref{eqn:roots1}.

Similarly, suppose now that $m = 5a$ for some $a$. We note that
\[
\cos\tfrac{2\pi}{5} = \cos\tfrac{8\pi}{5} = \frac{-1 + \sqrt{5}}{4} \quad \text{ and } \quad \cos\tfrac{4\pi}{5} = \cos\tfrac{6\pi}{5} = \frac{-1-\sqrt{5}}{4}
\]
and we see that
\[
\cos\tfrac{2\pi}{5} + \cos\tfrac{4\pi}{5} = \frac{-1 + \sqrt{5}}{4}  + \frac{-1 - \sqrt{5}}{4} = -\frac{1}{2}
\]
and
\[
2\,\cos\tfrac{2\pi}{5} \cos\tfrac{4\pi}{5} = \frac{\left(-1 + \sqrt{5}\,\right)\left(-1 - \sqrt{5}\,\right)}{8} = -\frac{1}{2}.
\]
Then, let $A = \{a, 4a\}$ and $B = \{2a, 3a\}$. For every choice of $j\in A$ and $\ell \in B$ (or vice versa), we obtain distinct solution $(j,\ell)$ to Equation \eqref{eqn:roots1}.
\end{proof}

Solutions of Equation \eqref{eqn:roots1} can be limited further by using an old result of Conway and Jones \cite{CoJo76}.

\begin{lemma}\label{lem:cossums} If $4 \nmid m$ and $4 \nmid m$, then \eqref{eqn:roots1} has only the trivial solution. Further, if $4 | m$, the only non-trivial solution is $\left\{ \frac{j}{m},\frac{\ell}{m} \right\} \subseteq \left\{ \frac{1}{4},\frac{3}{4} \right\}$. If  $ m = 5a$, the only non-trivial solutions are $(j,\ell) \in \{a, 4a\}\times\{2a, 3a\}\cup \{2a, 3a\}  \times \{a, 4a\}$.  \end{lemma}

\begin{proof}
We will make extensive use of \cite[Theorem 7]{CoJo76}, which gives the complete description of solutions to the equation:
\begin{equation}\label{eq:Conway}
A \cos 2\pi a + B \cos 2\pi b + C \cos 2 \pi c + D \cos 2 \pi d = E
\end{equation}
where all variables $A,\dots,E$ and $a,b,c,d$ are rational numbers. We may use an elementary trigonometric identity for cosine of angle sums to rewrite \eqref{eqn:roots1} as follows:
\begin{equation}\label{eqn:cos}
\cos x + \cos y =2 \cos x \cos y  = \cos(x+y) + \cos(x-y)
\end{equation}
where $x = 2\pi \frac{j}{m}$ and $y = 2 \pi \frac{\ell}{m}$. This is a special case of \eqref{eq:Conway} with $A=B=1$, $C=D=-1$ and $E=0$. Observe that if $y' = 2\pi - y$, then
\[
\cos(x + y') + \cos(x-y') = \cos(x + 2\pi - y) + \cos(x- 2\pi + y) = \cos(x-y) + \cos(x+y).
\]
Without loss of generality, we may therefore assume that $0 \leq x \leq y \leq \pi$. Let
\[
\C = \{ \cos x, \cos y, \cos(x+y) , \cos(y-x)\}.
\]
We will consider the following cases:
\begin{enumerate}[(i)]
\item There exists $c \in \C$ such that $c$ is rational.
\item $\C \cap \rats = \emptyset$ and some two elements of $\C$ have rational sum.
\item No proper subset of $\C$ has rational sum.
\end{enumerate}

\medskip
\noindent
{\bf Case (i).} It is known that if $\cos \theta \in \rats$ and $\theta$ is a rational multiple of $\pi$, then $\cos \theta \in \{\pm 1, \pm \frac{1}{2},0\}$.

We first suppose $\cos x \in \rats$.
If $\cos x = 1$, then (\ref{eqn:cos}) implies that $1 + \cos y = 2 \cos y$. Then, $x = y = 0$, giving the trivial solution.
If $\cos x = \frac{1}{2}$, then (\ref{eqn:cos}) implies that $\frac{1}{2} + \cos y = \cos y$, which is impossible. Similarly, if $\cos x = -\frac{1}{2}$, then (\ref{eqn:cos}) implies that $-\frac{1}{2} + \cos y =  - \cos y$ and $\cos y =  \frac{1}{4} \in \rats$, which is not possible.

If $\cos x = -1$, then (\ref{eqn:cos}) implies that $-1 + \cos y = -2 \cos y$, which gives that $\cos y = \frac{1}{3}$. Since $y$ is a rational multiple of $\pi$, this is impossible.
If $\cos x = 0$, then (\ref{eqn:cos}) implies that $\cos y = 0$.  In this case, $x,y \in \{\frac{\pi}{2}, \frac{3\pi}{2}\}$ and $\frac{j}{m}, \frac{\ell}{m} \in \{\frac{1}{4},\frac{3}{4}\}$. Then $4 | m$ and $\frac{j}{m}, \frac{\ell}{m} \in \{\frac{1}{4},\frac{3}{4}\}$ give all nontrivial solutions when some $c \in C$ is rational.

Since (\ref{eqn:cos}) is symmetric in $x$ and $y$, we may now assume that both $\cos x$ and $\cos y$ are irrational.

If both $\cos(x +y)$ and $\cos(x-y)$ are rational, then $\cos x + \cos y \in \rats$ and $\cos x, \cos y \notin \rats$. Note that $x, y \notin \{ 0, \pi, \frac{\pi}{2}\}$.  Let $\phi(\theta)$ be the following:
\[
\phi(\theta) = \begin{cases} \theta & \text{if } 0 < \theta < \frac{\pi}{2}; \\
							\pi - \theta & \text{if }  \frac{\pi}{2}< \theta < \pi; \\
							\theta - \pi & \text{if }  \pi < \theta < \frac{3\pi}{2}; \\
							 2\pi - \theta & \text{if }  \frac{3\pi}{2}< \theta < 2\pi. \end{cases}
\]
Then we note that $\cos(\phi(x)) \pm \cos(\phi(y)) \in \rats$ and $\cos(\phi(x)), \cos(\phi(y)) \notin \rats$.
We can apply Theorem 7 of \cite{CoJo76} to obtain that $ a\cos (\phi(x)) + b\cos(\phi(y)) = q$ for some $q \in \rats$ is proportional to $\cos\frac{\pi}{5} - \cos \frac{2\pi}{5} = \frac{1}{2}$. This implies that $5$ divides $m$ and we can see the only additional solutions are the ones found in Lemma \ref{lem:45cos}.

In the remainder of Case (i), we may assume that exactly one of  $\cos(x +y)$ and $\cos(x-y)$ is rational. Let $\theta$ be the corresponding argument. We may also assume that no rational linear combination of a proper subset of $\C \setminus \{\cos\theta\}$ has a rational sum. Let $\tau$ be the argument of the irrational element of $\{ \cos(x +y), \cos(x-y)\}$. Since cosinus is an even function, we may take $y-x$ instead of $x-y$ an assume that $0 \leq \tau, \theta \leq \pi$. We have that
\[
\cos x + \cos y - \cos \tau = \cos \theta \in \rats
\]
and let $a,b,c \in \{\pm1\}$ be such that
\begin{equation}\label{eqn:cos3terms}
a\cos (\phi(x)) + b\cos (\phi(y)) + c\cos (\phi(\tau)) = \cos \theta.
\end{equation}
Theorem 7 of \cite{CoJo76} gives that (\ref{eqn:cos3terms}) is proportional to one of the following:
\begin{enumerate}[(a)]
\item $-\cos \delta +\cos(\frac{\pi}{3} - \delta) + \cos(\frac{\pi}{3} + \delta) = 0$ for some $0 < \delta < \frac{\pi}{6}$;
\item $\cos\frac{\pi}{7} - \cos\frac{2\pi}{7} + \cos\frac{3\pi}{7} = \frac{1}{2}$;
\item $\cos\frac{\pi}{5} - \cos\frac{\pi}{15} + \cos\frac{4\pi}{15} = \frac{1}{2}$; and
\item $-\cos\frac{2\pi}{5} + \cos\frac{2\pi}{15} - \cos\frac{7\pi}{15} = \frac{1}{2}$.
\end{enumerate}

Observe that if (\ref{eqn:cos3terms}) is proportional to either (c) or (d), then $5$ divides $m$. If (\ref{eqn:cos3terms}) is proportional to (c), then we have that $\theta = \frac{\pi}{3}$, $\phi(\tau) = \frac{\pi}{15}$ and $\{\phi(x), \phi(y) \} = \{\frac{4\pi}{15}, \frac{5\pi}{15}\}$. We can see that there are no solutions for $x,y$ in this case. If (\ref{eqn:cos3terms}) is proportional to (d), then we have that $\theta = \frac{2\pi}{3}$, $\phi(\tau) = \frac{2\pi}{15}$ and $\{\phi(x), \phi(y) \} = \{\frac{6\pi}{15}, \frac{7\pi}{15}\}$. There are also no solutions for $x,y$ in this case.

To finish Case (i), we will show that (\ref{eqn:cos3terms}) cannot be proportional to either (a) or (b).

If (\ref{eqn:cos3terms}) is proportional to (a), then $\cos\theta = 0$ and thus $\theta \in \{\frac{\pi}{2}, \frac{3\pi}{2}\}$. We have three cases: $x+y = \frac{\pi}{2}$, $y-x = \frac{\pi}{2}$, or $x+y = \frac{3\pi}{2}$. If $\theta = x+y = \frac{\pi}{2}$, then $x\leq y \leq \frac{\pi}{2}$ so $\phi(z) = z$ for $z \in \{x,y, y-x\}$. Since (\ref{eqn:cos3terms}) is proportional to (a), $\{x,y,y-x\} = \{\delta, \frac{\pi}{3} \pm \delta \}$. If we sum all three elements, we obtain $x+y + y -x  =2y = \delta + \frac{2\pi}{3}$ and so $y = \frac{\delta}{2} + \frac{\pi}{3} \notin \{\delta, \frac{\pi}{3} \pm \delta \}$, a contradiction.

If $y-x = \theta = \frac{\pi}{2}$, then $x \leq \frac{\pi}{2} = \phi(x)$. If $x = \delta$, then $y = \frac{\pi}{2} + \delta$ and $\phi(y) = \pi -y = \frac{\pi}{2} - \delta$. We see that  $\frac{\pi}{2} - \delta \neq  \frac{\pi}{3} - \delta$, so $y = \frac{\pi}{3} + \delta = \frac{\pi}{2} - \delta$. Then $x = \phi(x) = \frac{\pi}{12}$. It follows that $4|m$, a contradiction. If $x = \frac{\pi}{3} - \delta$, then $y = \frac{\pi}{2} + x  = \frac{5\pi}{6} - \delta$ and $\phi(y) = \pi-y = \frac{\pi}{6} + \delta \notin \{\delta, \frac{\pi}{3} + \delta\}$, a contradiction. If $x = \frac{\pi}{3} + \delta$, then $y = \frac{\pi}{2} + x  = \frac{5\pi}{6} + \delta$ and $\phi(y) = \pi-y = \frac{\pi}{6} - \delta$. Then $\phi(y) = \delta$ and $\delta = \frac{\pi}{12}$. We obtain that $x = \frac{5\pi}{12}$ and so $4|m$, a contradiction.

If $x+y = \theta =\frac{3\pi}{2}$ then $y = \frac{3\pi}{2} -x$ and $\pi \geq x \geq \frac{\pi}{2}$ so $\phi(x) = \pi -x$. See Table \ref{tab:possvals} for all possible values of $x$, $y$ and $y-x$, based on the possible values of $\phi(x)$. If $\phi(x) = \frac{\pi}{3} + \delta$ then $\phi(y) = \frac{\pi}{6} - \delta$, which must equal $\delta$, hence $\delta = \frac{\pi}{12}$ and $y = \frac{13\pi}{12}$ and $4|m$, a contradiction. If $\phi(x) = \frac{\pi}{3} - \delta$ then $\phi(y) = \frac{\pi}{6} + \delta \notin \{\delta, \frac{\pi}{3} + \delta\}$, a contradiction. If $\phi(x) = \delta$ then $y-x  = - \frac{\pi}{2}  + 2\delta \geq 0$, since $y \geq x$. But $\delta \leq \frac{\pi}{6}$, so this cannot happen.
\begin{table}[h]
\centering
\begin{tabular}{|c|c|c|c|}
\hline
$\phi(x) = \pi - x$ & $\delta$ & $\frac{\pi}{3} - \delta$ & $\frac{\pi}{3} + \delta$ \\
\hline
$x$ & $\pi - \delta$ & $\frac{2\pi}{3} + \delta$ & $\frac{2\pi}{3} - \delta$ \\
$y$ & $\frac{\pi}{2} + \delta$ & $\frac{5\pi}{6} - \delta$  & $\frac{5\pi}{6} +\delta$ \\
$y-x$ & $-\frac{\pi}{2} + 2\delta$ & $\frac{\pi}{6} - 2\delta$  & $\frac{\pi}{6} +2\delta$ \\
\hline
\end{tabular}
\caption{Possible values of $\phi(x)$.}\label{tab:possvals}
\end{table}

If (\ref{eqn:cos3terms}) is proportional to (b), then $\cos\theta = \pm \frac{1}{2}$ and thus $\theta  \in \{\frac{\pi}{3}, \frac{2\pi}{3},\frac{4\pi}{3},\frac{5\pi}{3}\}$. Since $3$ and $7$ are prime, it is easy to see that any sum of two terms each of the form
\[
a \pi + b\,\tfrac{\pi}{7}
\]
where $a, b \in \ints$ will not be in the set $\{\frac{\pi}{3}, \frac{2\pi}{3},\frac{4\pi}{3},\frac{5\pi}{3}\}$. Thus, we cannot write $\theta$ as $x' +y'$ or $x'-y'$ for any $x',y'$ pre-images under $\phi$ of the arguments in (b), a contradiction.

\medskip
\noindent
{\bf Case (ii).} In this case, two of the elements of $C$ sum to a rational number, say $\cos\theta + \cos\tau$. Then, applying Theorem 7 of \cite{CoJo76} to $a\cos (\phi(\theta)) + b\cos (\phi(\tau))$ for $a,b \in \{\pm1\}$, we obtain that
\[a\cos (\phi(\theta)) + b\cos (\phi(\tau)) = \cos(\nicefrac{\pi}{5}) - \cos(\nicefrac{2\pi}{5})  = \frac{1}{2}.
\]
Here $5\mid m$ and we see again the only additional
solutions are the ones found in Lemma \ref{lem:45cos}.

\medskip
\noindent
{\bf Case (iii).} In this case, no proper subset $C$ has rational sum. We have that
\[
 \cos(\phi(x)) \pm \cos(\phi(y)) \pm \cos(\phi(x+y)) \pm \cos(\phi(x-y)) = 0
\]
(for some choice of the signs). We apply Theorem 7 of \cite{CoJo76} and see that this cannot hold, since all $4$ term sums in the theorem have non-zero sum. This completes the proof.
\end{proof}

\section{Computation on regular maps}\label{sec:app-computation}

We looked at the regular maps as given in the census of Marston Conder \cite{Con94, Conders}. For each map, we considered the map and its dual, and determined their vertex truncations. We restricted our computation to bipartite regular maps on at least $4$ vertices (since the 2-vertex case is completely covered in Section \ref{sect:genPetersen}) and with vertex (face) degree at least $3$. Note, the underlying graph of the multigraph does not necessarily have degree at least 3. Up to and including regular maps with $100$ edges, there are $351$ such regular maps. We determined the vertex truncations of these maps, as described in Section \ref{sec:regmaps}. Recall that the vertex truncation of a regular map with $e$ edges is a vertex-transitive cubic graph on $2e$ vertices that has $1$ as an eigenvalue. Amongst the $351$ considered vertex truncations of regular maps, $62$ have $1$ as a simple eigenvalue. Their properties are listed in Table \ref{tab:regmaps}.

\begin{center}
\begin{longtable}{|l|l|l|l|l|l|l|l|}
\caption{Vertex truncations of regular maps which have $1$ as a simple eigenvalue. The first seven columns give information about the regular map. The last column records if the truncation obtained is bipartite.  \label{tab:regmaps}}\\

\hline \multicolumn{1}{|c|}{\textbf{$|E|$}} & \multicolumn{1}{c|}{\textbf{$|V|$}} & \multicolumn{1}{c|}{\textbf{$|F|$}} &\multicolumn{1}{c|}{\textbf{Orientability}} &\multicolumn{1}{c|}{\textbf{Genus}} &\multicolumn{1}{c|}{\textbf{$d(v)$}}&\multicolumn{1}{c|}{\textbf{$d(f)$}} &\multicolumn{1}{c|}{\textbf{Truncation graph}}\\ \hline
\endfirsthead

\multicolumn{8}{c}%
{{\bfseries \tablename\ \thetable{} -- continued from previous page}} \\
\hline \multicolumn{1}{|c|}{\textbf{$|E|$}} & \multicolumn{1}{c|}{\textbf{$|V|$}} & \multicolumn{1}{c|}{\textbf{$|F|$}} &\multicolumn{1}{c|}{\textbf{Orientability}} &\multicolumn{1}{c|}{\textbf{Genus}} &\multicolumn{1}{c|}{\textbf{$d(v)$}}&\multicolumn{1}{c|}{\textbf{$d(f)$}} &\multicolumn{1}{c|}{\textbf{Truncation graph}}\\ \hline
\endhead

\hline \multicolumn{8}{|r|}{{Continued on next page}} \\ \hline
\endfoot

\hline \hline
\endlastfoot

12 & 8 & 6 & orientable & 0 & 3 & 4 &   \\
12 & 4 & 6 & orientable & 2 & 6 & 4 & bipartite \\
18 & 6 & 6 & orientable & 4 & 6 & 6 & bipartite \\
20 & 4 & 10 & orientable & 4 & 10 & 4 & bipartite \\
24 & 8 & 6 & orientable & 6 & 6 & 8 & bipartite \\
24 & 8 & 6 & orientable & 6 & 6 & 8 &   \\
24 & 16 & 6 & orientable & 2 & 3 & 8 &   \\
28 & 4 & 14 & orientable & 6 & 14 & 4 & bipartite \\
30 & 10 & 6 & orientable & 8 & 6 & 10 & bipartite \\
30 & 6 & 10 & orientable & 8 & 10 & 6 & bipartite \\
30 & 20 & 6 & non-orientable & 6 & 3 & 10 &   \\
36 & 4 & 18 & orientable & 8 & 18 & 4 & bipartite \\
36 & 24 & 12 & orientable & 1 & 3 & 6 &   \\
36 & 8 & 18 & orientable & 6 & 9 & 4 &   \\
36 & 12 & 6 & orientable & 10 & 6 & 12 & bipartite \\
40 & 8 & 10 & orientable & 12 & 10 & 8 & bipartite \\
42 & 6 & 14 & orientable & 12 & 14 & 6 & bipartite \\
42 & 14 & 6 & orientable & 12 & 6 & 14 & bipartite \\
44 & 4 & 22 & orientable & 10 & 22 & 4 & bipartite \\
48 & 16 & 6 & orientable & 14 & 6 & 16 & bipartite \\
48 & 32 & 16 & orientable & 1 & 3 & 6 &   \\
52 & 4 & 26 & orientable & 12 & 26 & 4 & bipartite \\
54 & 18 & 18 & orientable & 10 & 6 & 6 & bipartite \\
56 & 8 & 14 & orientable & 18 & 14 & 8 & bipartite \\
60 & 20 & 6 & orientable & 18 & 6 & 20 & bipartite \\
60 & 4 & 30 & orientable & 14 & 30 & 4 & bipartite \\
60 & 8 & 30 & orientable & 12 & 15 & 4 &   \\
60 & 8 & 20 & orientable & 17 & 15 & 6 &   \\
60 & 12 & 10 & orientable & 20 & 10 & 12 & bipartite \\
60 & 20 & 12 & non-orientable & 30 & 6 & 10 &   \\
60 & 40 & 12 & orientable & 5 & 3 & 10 &   \\
66 & 6 & 22 & orientable & 20 & 22 & 6 & bipartite \\
66 & 22 & 6 & orientable & 20 & 6 & 22 & bipartite \\
68 & 4 & 34 & orientable & 16 & 34 & 4 & bipartite \\
70 & 14 & 10 & orientable & 24 & 10 & 14 & bipartite \\
70 & 10 & 14 & orientable & 24 & 14 & 10 & bipartite \\
72 & 8 & 18 & orientable & 24 & 18 & 8 &   \\
72 & 24 & 12 & orientable & 19 & 6 & 12 &   \\
72 & 48 & 12 & orientable & 7 & 3 & 12 &   \\
72 & 8 & 18 & orientable & 24 & 18 & 8 & bipartite \\
72 & 24 & 18 & orientable & 16 & 6 & 8 & bipartite \\
72 & 16 & 18 & orientable & 20 & 9 & 8 &   \\
72 & 24 & 6 & orientable & 22 & 6 & 24 & bipartite \\
76 & 4 & 38 & orientable & 18 & 38 & 4 & bipartite \\
78 & 6 & 26 & orientable & 24 & 26 & 6 & bipartite \\
78 & 26 & 6 & orientable & 24 & 6 & 26 & bipartite \\
80 & 16 & 10 & orientable & 28 & 10 & 16 & bipartite \\
84 & 8 & 28 & orientable & 25 & 21 & 6 &   \\
84 & 12 & 14 & orientable & 30 & 14 & 12 & bipartite \\
84 & 8 & 42 & orientable & 18 & 21 & 4 &   \\
84 & 28 & 6 & orientable & 26 & 6 & 28 & bipartite \\
84 & 4 & 42 & orientable & 20 & 42 & 4 & bipartite \\
88 & 8 & 22 & orientable & 30 & 22 & 8 & bipartite \\
90 & 18 & 10 & orientable & 32 & 10 & 18 & bipartite \\
90 & 10 & 18 & orientable & 32 & 18 & 10 & bipartite \\
90 & 30 & 6 & orientable & 28 & 6 & 30 & bipartite \\
92 & 4 & 46 & orientable & 22 & 46 & 4 & bipartite \\
96 & 32 & 24 & orientable & 21 & 6 & 8 &   \\
96 & 64 & 24 & orientable & 5 & 3 & 8 &   \\
96 & 32 & 6 & orientable & 30 & 6 & 32 & bipartite \\
98 & 14 & 14 & orientable & 36 & 14 & 14 & bipartite \\
100 & 4 & 50 & orientable & 24 & 50 & 4 & bipartite \\

\end{longtable}
\end{center}

One can make some basic observations. The graphs of type $C(m,k)$ are obtained when $k$ is the vertex degree of the regular map. From these examples it looks that $k$ is always even or divisible by 3. Of course, this could be the case because the graphs are too small and because we have excluded the two-vertex case (which gives the generalized Petersen graphs). Note that the 2-vertex case includes $k=5$ with two examples: the 5-prism (which has 1 as simple eigenvalue) and the Petersen graph (where 1 is not simple). More generally, for every odd $k$, the $k$-prism has 1 as simple eigenvalue. For $k=3$, there are infinitely many examples of type $C(\cdot,3)$ (see Section \ref{subsec:truncation}), but it is not known whether there are infinitely many examples of type $C(\cdot,k)$ for any other $k$.


\end{document}